\documentclass{article}
\usepackage{amsmath,amssymb,amsthm}
\usepackage{a4wide}
\usepackage{color}
\usepackage{enumerate}
\usepackage{cleveref}

\newtheorem{thm}{Theorem}[section]
\newtheorem*{thm*}{Theorem}
\newtheorem*{conj*}{Conjecture}
\newtheorem{cor}[thm]{Corollary}

\newtheorem{lem}[thm]{Lemma}

\newtheorem{lemma}[thm]{Lemma}
\newtheorem{prop}[thm]{Proposition}

\theoremstyle{remark}

\newtheorem{remark}[thm]{Remark}
\newtheorem{notn}[thm]{Notation}
\theoremstyle{definition}

\newtheorem{defn}[thm]{Definition}
\newcounter{claim}[thm]

\newcommand{\calS}{\mathcal{S}}
\newcommand{\prodS}{\mathrm{Prod}}

\DeclareMathOperator{\Alt}{Alt}

\DeclareMathOperator{\Sym}{Sym}
\DeclareMathOperator{\Aut}{Aut}
\DeclareMathOperator{\Out}{Out}

\newcommand{\PGL}{\mathrm{PGL}}
\newcommand{\GL}{\mathrm{GL}}
\newcommand{\PSL}{\mathrm{PSL}}
\newcommand{\SL}{\mathrm{SL}}

\newcommand{\Sp}{\mathrm{Sp}}

\newcommand{\PSU}{\mathrm{PSU}}
\newcommand{\PSp}{\mathrm{PSp}}
\newcommand{\POmega}{\mathrm{P}\Omega}
\newcommand{\GU}{\mathrm{GU}}
\newcommand{\SU}{\mathrm{SU}}
\newcommand{\wh}{\widehat}

\title{Finite simple groups have many classes of $p$-elements}

\author{Michael Giudici, Luke Morgan and  Cheryl E.~Praeger\\Department of Mathematics and Statistics\\
The University of Western Australia, Perth, Australia
}

\begin{document}

\maketitle

%
%

\begin{abstract}
 For an element $x$ of a finite group $T$, the $\Aut(T)$-class of $x$ is the set $\{ x^\sigma\mid \sigma\in \Aut(T)\}$. We prove that the order $|T|$ of a finite nonabelian simple group $T$ is bounded above by a function of the parameter $m(T)$, where $m(T)$ is the maximum, over all primes $p$, of the number of $\Aut(T)$-classes of elements of $T$ of $p$-power order.  This bound is a substantial generalisation of results of Pyber, and of H\'ethelyi and K\"ulshammer, and it has implications for relative Brauer groups of finite extensions of global fields.  


    \bigskip
        \emph{Dedicated with admiration and thanks to the memory of our colleague Gary M.~Seitz.}
\end{abstract}

\section{Introduction}
In 1992 Laci Pyber \cite{Py92} showed that a group of order $n$ contains at least $O(\log n/(\log\log n)^8)$ conjugacy classes of elements. This solved a problem of Brauer from 1963 \cite{Br63}, who had asked for a significant improvement on his lower bound of $\log\log n$. In Subsection~\ref{s:discussion} we briefly discuss the interesting story around these bounds which date back to work of Landau in 1903 and extend to recent work in 2017.  The special case of Pyber's bound for a nonabelian simple group $T$ could be turned around to state that $|T| < c^{f(m)}$ where $m$ is the number of $\Aut(T)$-classes in $T$, $c$ is a constant, and $f$ is the particular function $f(m)=(\log m)^2\cdot \log\log m$. This alternative statement of Pyber's result was used in  \cite[Theorem 4.4]{kron} to prove a conjecture about maximal subgroups of a finite group which are `covering subgroups', and in turn, this application had consequences for Kronecker classes of algebraic number fields (see \cite[Section 4]{kron}). 

Many classical results concerning conjugacy classes of elements in groups have analogues in the case where the conjugacy classes are restricted to those consisting of elements of prime power order. For example, given a core-free subgroup $H$ of a group $G$, not only is a there a conjugacy class of elements of $G$ that does not meet $H$,   but Fein, Kantor and Schacher  \cite{FKS} show that there is a conjugacy class of elements of prime power order that avoids $H$. Similarly, not only does every nonlinear irreducible character of a finite group vanish on some conjugacy class of elements, but Malle, Navarro and Olsson \cite{MNO} show that each such character must vanish on some conjugacy class of elements of prime power order. Such analogues usually have interesting applications: for example, the Fein-Kantor-Schacher result  is equivalent (see \cite[\S3]{FKS}) to the fact that the relative Brauer group of a nontrivial finite extension of global fields is infinite.

In this paper we prove  a new bound (Theorem~\ref{t:main}) on the order of a finite simple group related to its $p$-elements, that is, elements of $p$-power order for various primes $p$. The bound is a substantial generalisation of the results of Pyber and others, in that the parameter $m$ above is replaced by 
 \begin{equation}\label{mT}
m(T)=\max_{\text{primes}\ p} m_p(T), \ \text{where} \  m_p(T) = \#\{ \mbox{$\Aut(T)$-classes of elements of $p$-elements in $T$}\}.
\end{equation} 
\begin{thm}\label{t:main}
There exists an  increasing  function $f$ on the natural numbers such that, for a finite nonabelian simple group $T$, the  order of $T$ is at most $f(m(T))$. 
\end{thm}

In the proof of Theorem~\ref{t:main} for exceptional groups of Lie type, we are indebted to the work of Gary Seitz, to whose memory this paper is dedicated, and his co-authors Martin Liebeck and Jan Saxl, for their classification of the subgroups of maximal rank of these groups \cite{LiebeckSaxlSeitz}. Their results gave us the detailed information about certain tori and their normalisers on which our proof is based.   In addition, in the proof of Theorem~\ref{t:main} for classical groups, we use the description by Aschbacher and Seitz \cite{aschseitz} of conjugacy classes of involutions in Chevalley groups of even characteristic to bound the dimension (Lemma~\ref{lem:unipbound}).


 The function $f(n)$ we obtain in the proof of Theorem~\ref{t:main} involves an $n!$ term. It is possible that a better function might be obtained, see Remark~\ref{rem:on the function} for further comments.
 Since the function $f(n)$ in   Theorem~\ref{t:main} is increasing, the bound  can be turned around to give a lower bound in terms of $|T|$ for the number of $\Aut(T)$-classes of $p$-elements in $T$.

\begin{cor}\label{c:main}
    There exists an increasing function $g$ on the natural numbers such that, for a finite nonabelian simple group $T$, there exists a prime $p$ dividing $|T|$ such that the number of $\Aut(T)$-classes of elements of $p$-power order in $T$ is at least $g(|T|)$.
\end{cor}


There are numerous bounds in the literature that relate $|T|$ with various parameters concerning numbers of conjugacy classes or $\Aut(T)$-classes. For example, by \cite[Theorems 1.2 and 1.4]{MN16},  for a given prime $p$ dividing $|T|$, the order $|T|$ is bounded above in terms of the number of its $p$-regular conjugacy classes (elements of order coprime to $p$) and also, apart from certain rank 1 Lie type simple groups,  $|T|$ is bounded above in terms of the number of its $p$-singular conjugacy classes (elements of order a multiple of $p$).  
One motivation for proving Theorem~\ref{t:main} is a conjecture concerning finite groups $G$ with a  proper subgroup that meets all $\Aut(G)$-classes of elements of $G$ of prime power order \cite[Conjecture $4.3'$]{kron}. Thus, rather than considering $p$-singular or $p$-regular elements for some fixed prime $p$, we must work with all elements of prime power order. In future work \cite{GMP} we apply Theorem~\ref{t:main} to prove an important case of \cite[Conjecture $4.3'$]{kron}, which has consequences for relative Brauer groups of field extensions as discussed in \cite{FKS, GMP, G90}.
%

Finally we note that many bounds of this type in the literature are available for general finite groups, and it would be interesting to know if the bounds in Theorem~\ref{t:main} and Corollary~\ref{c:main} can be used to obtain similar bounds for larger families of finite groups.

This paper is organised  as follows. In Section~\ref{s:prod} we prove some preliminary numerical results and a result relating the normalisers of cyclic subgroups $S$ of a group $G$ with our parameter $m(G)$ \eqref{mT}. We treat the alternating groups in Section~\ref{sec:alt} and the bulk of the work takes place in Sections~\ref{sec:class} and \ref{sec:excep} where we consider the classical groups and the exceptional groups of Lie type, respectively. Finally in Section~\ref{sec:proof} we complete the proof of Theorem~\ref{t:main}.

\subsection{Commentary on Landau's and Pyber's theorems}\label{s:discussion}

Landau's theorem~\cite{La} from 1903 states that, for a given positive integer $k$, there are only finitely many finite groups having exactly $k$ conjugacy classes of elements, and so such a group must have order bounded in terms of $k$. Brauer~\cite{Br63} made this bound explicit in 1963, showing that a group of order $n$ has at least $\log\log n$ conjugacy classes, and asked for a substantially better bound. Providing an improvement was the main focus of Pyber's 1992 paper \cite{Py92}, where he proved that a group of order $n\geqslant  4$ must have at least $c\log n/(\log\log n)^8$ conjugacy classes for some `computable constant' $c$. The  relevance for this paper is Pyber's bound for nonabelian simple groups  \cite[Lemma 4.4]{Py92}: if $T$ is a finite nonabelian simple group and $a=|\Aut(T)|$, then the number of $\Aut(T)$-classes in $T$ is at least 
    \[
    2^{c(\log a/\log\log a)^{1/2}},\quad \text{for some constant $c$.}
    \]   
Since 1992 there have been numerous contributions that strengthened these bounds (see \cite{BMV} for an overview). Currently the best lower bound for the number  $k(G)$ of conjugacy classes of an arbitrary finite group $G$ is given by Baumeister et al in 2018, \cite[Theorem 1.1]{BMV}:
\[
\forall\ \epsilon>0, \ \exists\ \delta>0\ \ \text{such that, $\forall$\  finite groups $G$ with  $|G|\geqslant  3$, }\ k(G)\geqslant  \delta \log |G|/(\log\log |G|)^{3+\epsilon}.
\]
There are many better lower bounds available for restricted classes of groups, for example, if $G$ is soluble then Keller~\cite{K11} proved that $ k(G)\geqslant  c \log |G|/(\log\log |G|)$ for some constant $c$; a purely logarithmic lower bound $k(G)>\log_3 |G|$ was given in \cite[Theorem 1.2]{BMV} for groups with trivial soluble radical; and a better than logarithmic lower bound was obtained for nilpotent groups by Jaikin-Zapirain~\cite{Za}. These bounds have been exploited to obtain related bounds concerning irreducible complex representations, for example, bounding the number of irreducible characters of odd degree \cite{GY24, HKY21} in connection with the McKay conjecture. Also, as mentioned above, there are various results in \cite{MN16} that give lower bounds for the number of conjugacy classes of $p$-regular elements, or $p$-singular elements, or the total number of classes of elements of prime-power order (adding over the prime divisors) in \cite{HethKul}, but to our knowledge our bound in terms of classes of $p$-elements, for a certain single prime $p$, is new.  

\begin{remark}\label{rem:mexp}
 Our proof of Theorem~\ref{t:main} for simple classical groups    hints towards a possible bound for the order of a simple group in terms of another property.   For a group $G$ and a prime $s$ we define
\[m_{s\text{-exp}}(G) = \# \{ \Aut(G)\text{-classes of elements of order }\exp(G)_s\}
\]
where $\exp(G)_s$ is the $s$-part of the exponent  $\exp(G)$,
and set $m_{\text{exp}}(G)$ to be the maximum of $m_{s\text{-exp}}(G)$ over all primes $s$ dividing $|G|$. In addition, for a set $\mathcal S$ of primes  we set $m_{\mathcal S\text{-exp}}(G)$ to be the maximum of $m_s(G)$ over all primes $s \in \mathcal S$.

The above concepts are motivated by choices made in the proof of Proposition~\ref{prop:class}. It turns out that  for simple classical groups $T$, apart from the characteristic $p$, the set $\mathcal{S}(T)$ of primes $s$ we consider all have the property that the Sylow $s$-subgroups are cyclic and hence have order $\exp(T)_s$, and the  $s$-elements we consider are those of maximal order $\exp(T)_s$. 
Thus for a simple classical  group $T$ of characteristic $p$, we prove that $|T|$ is bounded above by a function of 
\[
m'(T)=\max\{m_p(T), m_{\mathcal{S}(T)\text{-exp}}(T)\}
\]
where $m_p(T)$ is as in \eqref{mT}. See Remark~\ref{rem: bounding mexpt for classical} for further details.
 This motivated us to consider whether a similar bound holds for other simple groups. In Remark~\ref{rem:mexp-alt} we show, using the prime number theorem, that for large enough $m$, $|\Alt(m)|$ is bounded above by a function of $m_\text{exp}(\Alt(m))$. 
 We believe this style of bound has not previously been studied. It would be interesting to know if such a bound holds also for the simple exceptional  groups of Lie type. 
%
 %
\end{remark}




\subsection{Acknowledgements}
The third author thanks Laci Babai for early discussions about an approach to proving the bound in Theorem~\ref{t:main} for rank one Lie type groups. 
 This research was supported by the Australian Research Council Discovery Project grant DP230101268.  
We thank the referee for their comments on our manuscript.


%
%
%
%
%
%
%
%
%
%

\section{Preparatory lemmas}
\label{s:prod}

For a prime $s$ and integer $n$, let $n_s$ denote the $s$-part of $n$, that is,
the highest power of $s$ dividing~$n$ and let $n_{s'}$ denoted the part of $n$ prime to $s$, that is,   that is, $n/n_s$.  For an integer $m$ define $\mathcal{S}(m)$ to
be the set of all prime divisors $s$ of $q^m-1$ such that $(q^m-1)_s$ does not
divide $q^t-1$ for any $t<m$.  In our analysis we need that
\[
\prodS(m,q):=\prod_{s\in\mathcal{S}(m)}(q^m-1)_s
\]
 is large enough, for example, larger than $q$ or
some constant multiple of $q$.  We use the following
lemma to treat the  classical groups.  

\begin{lemma}\label{lem:ppd} 
Let $q$ be a prime power and $m$ a positive integer. 
\begin{description}
\item[(i)]Suppose that $s$ is a prime divisor of $q^m-1$. If $t$ is the 
smallest positive integer such that $(q^m-1)_s$ divides $q^t-1$, then $t$
divides $m$.
\item[(ii)] $q+1$ divides $\prodS(2,q)$, and $q^2+1$ divides $\prodS(4,q)$.

\item[(iii)] If $m$ is odd and $m>1$, then
$2\not\in\mathcal{S}(2m)$ and $\prodS(2m,q)$ divides 
$q^m+1$.

\item[(iv)] Suppose that $m$ is an odd prime. Then 
\begin{enumerate}
    \item[(1)]  if $s\in\calS(m)$ and $s\mid q-1$, then $s=m$; if $s\in\calS(2m)$ and $s\mid q^2-1$, then $s=m$ and $m\mid q+1$;
    \item[(2)]    $(q^m-1)/(q-1)$ divides  $\prodS(m,q)$; and 
    \item[(3)]  $(q^m+1)/(q+1)$ divides $\prodS(2m,q)$.    
\end{enumerate}

\end{description}
\end{lemma}

\begin{proof}
(i) Let $s^b=(q^m-1)_s$ and let $t$ be the least positive integer such that
$s^b$ divides $q^t-1$. Then $t\leqslant m$. Write $m=kt+r$, where $1\leqslant r\leqslant t$.
Then $s^b$ divides $\gcd(q^m-1, q^t-1) = q^{\gcd(m,t)}-1$, and $\gcd(m,t)
=\gcd(r,t)\leqslant r\leqslant t$. By the minimality of $t$ we have $r=t$, and hence 
$t$ divides $m$.

(ii) If a prime $s\mid q^2-1$ then either $s\in\mathcal{S}(2)$ or 
$(q^2-1)_s$ divides $q-1$, and hence $q^2-1$ divides $(q-1)\cdot \prodS(2,q)$, so  $\prodS(2,q)$ is divisible by $q+1$. Similarly, if a prime $s\mid q^4-1$ then, by part (i), either $s\in\mathcal{S}(4)$ or $(q^4-1)_s$ divides $q^2-1$, and 
hence $q^2+1$ divides $\prodS(4,q)$.

(iii) Next assume that $m$ is odd and $m>1$. If $q$ is even then $2\not\in\mathcal{S}(2m)$ by the definition of $\mathcal{S}(2m)$. It turns out that this also holds if $q$ is odd: for by \cite[Lemma 2.5]{GueP}, 
\[
(q^{2m}-1)_2 = ((q^2)^m-1)_2 = (q^2-1)_2
\]
and as $m>1$, again $2\not\in\mathcal{S}(2m)$ by the definition of $\mathcal{S}(2m)$. Thus, for any $q$, $\mathcal{S}(2m)$ consists of odd primes. Since $\gcd(q^m-1,q^m+1)=(2,q-1)$, it follows that for any odd prime $s$, $(q^{2m}-1)_s$ divides exactly one of $q^m+1$ and $q^m-1$.  Moreover, if $s\in\mathcal{S}(2m)$ then $(q^{2m}-1)_s$ does not divide $q^m-1$ by definition, and hence, for each $s\in\mathcal{S}(2m)$, we have that  $(q^{2m}-1)_s$ divides $q^m+1$. 
Hence $\prodS(2m,q)$ divides $q^m+1$, proving (iii).

(iv) Now assume that $m$ is an odd prime and let $s$ be a prime dividing $q^m-1$. Then either $s\in\mathcal{S}(m)$ or $(q^m-1)_s$ divides 
$q-1$, and hence $q^m-1$ divides $(q-1)\cdot \prodS(m,q)$, so  $\prodS(m,q)$ is divisible by $(q^m-1)/(q-1)$. Further, if $s\in\calS(m)$ and $s\mid q-1$, then $s$ must divide $(q^m-1)/(q-1)$ by the definition of $\calS(m)$, and hence $s$ divides $\gcd((q^m-1)/(q-1), q-1) = \gcd (m,q-1)$. Since $m$ is prime this implies that $s=m$. 
Similarly if $s\in\calS(2m)$ and $s\mid q^2-1$, then $s$ must divide $(q^{2m}-1)/(q^2-1)$ by the definition of $\calS(2m)$, and hence $s$ divides $\gcd((q^{2m}-1)/(q^2-1), q^2-1) = \gcd (m,q^2-1)$. Since $m$ is prime this implies that $s=m$ divides $q^2-1$. Further, by part (iii), $s\mid q^m+1$ and hence $s$ divides $\gcd(q^m+1, q^2-1) = q+1$. Thus parts (1) and (2) are proved.

Finally, if a prime $s$ divides $q^{2m}-1$ and $t$ is 
minimal such that $(q^{2m}-1)_s$ divides $q^t-1$, then either $s\in\mathcal{S}(2m)$, 
or $t\in\{m,2,1\}$. Thus $q^{2m}-1$ divides 
\[
\prodS(2m,q)\cdot \frac{(q^m-1)(q^2-1)}{\gcd(q^m-1,q^2-1)} =  \prodS(2m,q)\cdot (q^m-1)(q+1).
\]
It follows that $(q^m+1)/(q+1)$ divides $\prodS(2m,q)$.
\end{proof}

    


\begin{defn}
    We say that a group $G$ is \emph{prime power bounded by $n$}, or simply, \emph{pp-bounded by $n$}, if $m(G) \leqslant n$, where $m(G)$ is as in~\eqref{mT}.
   In other words, for each prime $p$ dividing $|G|$, the number of $\Aut(G)$-classes of elements of $p$-power order in $G$ is at most $n$.
\end{defn} 

In the following lemma $\phi$ denotes the Euler $\phi$-function, namely for a positive integer $m$, $\phi(m)$ is the number of positive integers at most $m$ and coprime to $m$.

\begin{lem}
\label{lem: bounding phi s}
Suppose that $G$ is a group  that is pp-bounded by $n$. Let $S \leqslant G$ be a nontrivial cyclic $s$-subgroup   of order $s^b$, where $s$ is prime,  and let $N=N_{\Aut(G)}(S)$. 
Then there is a bijection $\mathcal{C} \rightarrow \mathcal{D}$, where
\[
\mathcal C  := \{ x^{\Aut(G)} \mid x\in S, o(x) = s^b \},\quad  
 \mathcal D:=\{ x^N  \mid x\in S, o(x) = s^b \},
\]
and  $|\mathcal C |=\phi(|S|)/r $, with $r=|N_{\Aut(G)}(S) : C_{\Aut(G)}(S)|$. Furthermore, 
\[
\phi(|S|)=s^{b-1}(s-1) \leqslant  rn \quad \text{and} \quad \phi(|S|) \text{ divides }  r(n!).
\]
\end{lem}
\begin{proof}
 Note that if $x_1,x_2\in S$ have order $s^b$ and are such that $x_1^g = x_2$ for some $g\in \Aut(G)$, then $S^g = \langle x_1\rangle^g = \langle x_1^g \rangle = \langle x_2 \rangle = S$, so that $g\in N$. Conversely, if $x_1, x_2 \in S$ are conjugate by an element of $N$, then they are also conjugate under the action of $\Aut(G)$. 
This gives a bijection $\mathcal{C} \rightarrow \mathcal{D}$ as claimed.

For $x^N \in \mathcal D$, we know $\langle x \rangle = S$, so  $|x^N | = |N:C_N(x)| = |N:C_N(S)|$. Note that $C_{\Aut(G)}(S) \leqslant N_{\Aut(G)}(S)=N$ so that $C_{\Aut(G)}(S)=C_N(S)$. Hence $|N:C_N(S)| =r $. Thus each $N$-class in $\mathcal D$ has length exactly $r$.
Since  $S$ contains $\phi(|S|)$ elements of order $|S|$, where  $\phi$ is Euler's function, we have that $r$ divides $\phi(|S|)$ and $|\mathcal D| = \phi(|S|)/r$. Since $G$ is pp-bounded by $n$,
$ \phi(|S|)/r = |\mathcal D | \leqslant n$, 
and hence in particular $\phi(|S|)$ divides $r(n!)$.
\end{proof}

\begin{notn}
Throughout this  paper $\log(x)$ denotes the   natural logarithm of $x$.
\end{notn}

\begin{lemma}
    \label{lem: sqrtx over log x}
    For all $a\geqslant 1$ we have 
    \[
    \frac{a^{1/4}}{4} < \frac{\sqrt{a}}{ \log(a)}.
    \]
\end{lemma}
\begin{proof}
The inequality in the statement is equivalent to $\log(a) < 4  a^{1/4}$. Exponentiating both sides, we see the latter inequality is equivalent to $a < e^{4  a^{1/4}}$, that is, 
\[
f(a) := e^{  4  a^{1/4}} - a > 0.
\]
Now  $f'(a) = \frac{e^{ 4 a^{1/4}}}{a^{3/4}} -1$ and  $f''(a) = \frac{e^{ 4 a^{1/4}}}{a^{3/2}}(1-\frac{3}{4a^{1/4}})$. For $a\geqslant 1 $ we have  $f''(a) >0$, so that $f'(a)$ is increasing for $a\geqslant 1$. Further, $f'(1) = e^{4} - 1 >0$, so $f'(a)>0$ for all $a\geqslant 1$. Thus $f(a)$ is increasing for $a\geqslant 1$, and since $f(1) = e^4-1>0$ also, we have $f(a) > 0$ for all $a \geqslant 1$, as required.
\end{proof}

\begin{lemma}
\label{lem: log bound}
    Suppose that $q=p^a$, $p\geqslant   2$ and $a \geqslant 1$. Then
    \[ \frac{\log(2) \sqrt{q}}{2} \leqslant \frac{q}{a}.\]
\end{lemma}
\begin{proof}
    Note that $a  = \frac{\log(q)}{\log(p)} \leqslant \frac{\log(q)}{\log(2)}$. Now since $\log(x) \leqslant 2\sqrt{x}$ for all $x\in \mathbb R$ with $x>0$, we have $a \leqslant \frac{2q^{1/2}}{\log(2)}$, and thus $a q^{1/2} \leqslant \frac{2q}{\log(2)}$ which yields $\frac{\log(2)q^{1/2}}{2} \leqslant \frac{q}{a}$.
\end{proof}

\section{The alternating groups}
\label{sec:alt}

We first consider the alternating groups.

\begin{lemma}
  \label{lem: alt}  
  If $T\cong \Alt(m)$ is $pp$-bounded by $n$, then $|T| \leqslant \frac{1}{2}(3n+2)!$.
\end{lemma}
\begin{proof}
    Suppose that $T\cong \Alt(m)$ is $pp$-bounded by $n$. If $m \neq 6$, we note that $T$ has $\lfloor m/3 \rfloor$ $\Aut(T)$-classes of elements of elements order $3$. This gives $m \leqslant 3n+2$, and hence $|T|=\frac{1}{2}\cdot m!\leqslant \frac{1}{2}\cdot (3n+2)!$. If $m=6$, then $T$ has $2$ classes of $2$-elements, so $n\geqslant 2$. Then certainly  $|T|=360 \leqslant \frac{1}{2}\cdot (3n+2)!$ holds.
\end{proof}

\begin{remark}
    \label{rem:mexp-alt}
 As discussed in Remark~\ref{rem:mexp}, another type of  bound on $|T|$ for a simple group $T$
is in terms of the number of $\Aut(T)$-classes of elements with order equal to  $\exp(T)_p$, the $p$-part of the exponent of $T$. We explain briefly that there is such a bound when $T\cong \Alt(m)$, for sufficiently large $m$. Recall that $\exp(\Alt(m)) =\mathrm{lcm}\{ m' : 1 \leqslant m' \leqslant m \}$.  

Let $\pi(x)$ denote  the number of   integers less than or equal to $x$  that  are prime. Based on the Prime Number Theorem, we know   (for  sufficiently large $x$) that
\begin{equation}
\label{eq:prime number theorem}
c \frac{ x}{\log x } < \pi(x) < d \frac{ x}{\log  x } 
\end{equation}
for explicit  constants $c$ and $d$; for example, if $x\geqslant 11$, then we can take $c=1$, and for $x$ `sufficiently large' we can take $d=1.04423$, see \cite[pp.~176--177]{R}. These bounds allow us to prove existence of primes in certain intervals, somewhat analogously to Bertrand's Postulate which asserts that there is a prime between $m$ and $2m$ for for all positive integers $m$. We are interested here in the existence of a prime $p$  such that $\sqrt{m/3} < p \leqslant \sqrt{m/2}$; and we claim that such a prime exists for sufficiently large $m$. 
%
Indeed, if $\sqrt{m/3}$ is large enough so that \eqref{eq:prime number theorem} holds with $c=1$, then we have 
\begin{center}
    $\pi\left(\sqrt{m/2}\right) >  \frac{\sqrt{m/2}}{\log\left(\sqrt{m/2}\right)}$ and $\pi\left(\sqrt{m/3}\right) < d \frac{\sqrt{m/3}}{\log\left(\sqrt{m/3}\right)}$,
\end{center}
and a sufficient condition for $p$ to exist is that
\[ 
\frac{\sqrt{m/2}}{\log\left(\sqrt{m/2}\right)} > d \frac{\sqrt{m/3}}{\log\left(\sqrt{m/3}\right)}
\]
which is equivalent to 
\[
\log(m) > \frac{\log(3)/\sqrt{2} - d \log(2)/\sqrt{3}}{1/\sqrt{2}-d/\sqrt{3}}
\]
and since $d$ is a constant, this holds for sufficiently large $m$.

Thus, provided $m$ is large enough we may choose a prime $p$ such that $\sqrt{m/3} < p \leqslant \sqrt{m/2}$, or equivalently,  $2p^2\leqslant m < 3p^2$. We may assume that $p>3$ (by taking $m>27$) and then we see that the $p$-part $\exp(T)_p$ of the exponent is exactly  $p^2$. Meanwhile, the inequality   $2 p^2 \leqslant m$ implies that   $T$ has at least  $p$ $\Aut(T)$-classes of elements of order $p^2$, namely, for  $1\leqslant i \leqslant p$, elements that have a single  $p^2$-cycle and exactly $i$ cycles of length $p$.  
Such elements with different values of $i$ cannot be conjugate in $\Aut(T)\cong \Sym(m)$ because they have different cycle  types, so there are at least $p$ $\Aut(T)$-classes of elements with exponent $\exp(T)_p$, that  is,  $m_{\text{exp}}(T)\geqslant  p$. Thus $m< 3p^2 \leqslant 3(m_{\text{exp}}(T))^2$, and hence $|T|$ is bounded above by $\frac{1}{2} (3 (m_{\text{exp}}(T))^2)!$. 
\end{remark}

\section{Finite simple classical groups}
\label{sec:class}

Let $T$ be a finite simple classical group    defined over a field of order  $q=p^a$, as in one of the lines of the following table. Table~\ref{tab: classical gps} records also the natural module $V$ for $T$ and its dimension, and the covering group $G$ of $T$ in $\GL(V)$. 

\begin{table}[ht]
\begin{center}
\begin{tabular}{cclc |c}
\hline
  $T$&$d$ &$V$ & $G$ & Conditions  \\ \hline
   $\PSL(d,q)$  & $d\geqslant  2$ & $\mathbb{F}_q^d$ & $\SL_d(q)$ & $(d,q)\ne(2,2)$ or $(2,3)$\\ 
   $\PSU(d,q)$  & $d\geqslant  3$ & $\mathbb{F}_{q^2}^d$ & $\SU_d(q)$ & $(d,q)\ne(3,2)$\\ 
    $\PSp(2d,q)$  & $d\geqslant  2$ & $\mathbb{F}_q^{2d}$ & $\Sp_{2d}(q)$ & $(d,q)\ne(2,2)$\\ 
   $\POmega^\circ(2d+1,q)$  & $d\geqslant  3$ & $\mathbb{F}_q^{2d+1}$ & $\Omega_{2d+1}^\circ(q)$ & $q$ odd\\  
      $\POmega^\epsilon(2d,q)$  & $d\geqslant  4$ & $\mathbb{F}_q^{2d}$ & $\Omega_{2d}^\epsilon(q)$ & $\epsilon\in\{+,-\}$\\ \hline
   \end{tabular}
   \caption{The simple classical groups, their  covering groups and natural modules}
   \label{tab: classical gps}
\end{center}
\end{table}

\begin{lem}\label{lem:unipbound}
Let $T$ and $d$ be  as in Table~\ref{tab: classical gps}. Then  the number of unipotent  conjugacy classes is at least $d$.
\end{lem}
\begin{proof}
Note that the identity element forms a unipotent conjugacy class, and since there is always a non-identity unipotent conjugacy class, the assertion holds if $d=2$. Thus we may assume that $d\geqslant 3$. Furthermore, making use of the identity, we need to show the existence of $d-1$ nonidentity unipotent classes.
The unipotent conjugacy classes in classical groups are well known, see for example \cite{BGbook,DLO,lieseitzunipot,wall}, or \cite{aschseitz} for unipotent elements of order 2.  

Let $J_i$ denote an $i\times i$ Jordan block with all 1's down the diagonal. Then for each $i$ such that $2\leqslant i\leqslant d$,  the group $T=\PSL(d,q)$ contains an element corresponding to the matrix that is the direct sum of $J_i$ and $I_{d-i}$. These are clearly not conjugate under any element of $\mathrm{P}\Gamma\mathrm{L}(d,q)$ or under the inverse transpose map, so the number of $\Aut(T)$-conjugacy classes of nonidentity unipotent elements is at least $d-1$. Each of these conjugacy classes meets $\PSU(d,q)$ nontrivially and so for $T=\PSU(d,q)$ the number of $\Aut(T)$-conjugacy classes of nonidentity unipotent elements is also at least $d-1$.


Suppose now that $T= \PSp(2d,q)$ (with  $d\geqslant 3$).  For each $i\leqslant d$, the group $T$ contains an element whose corresponding matrix has Jordan canonical form that is the direct sum of $i$ copies of $J_2$ and $I_{2(d-i)}$. This gives us $d$ different conjugacy classes in $T$, that are clearly not fused in $\mathrm{P}\Gamma\mathrm{Sp}(2d,q)$. 
Since   $2d >4$ we have  $\Aut(T)=\mathrm{P}\Gamma\mathrm{Sp}(2d,q)$ and the result is proved.

Next suppose that $T=\POmega^\circ(2d+1,q)$ and $q$ is odd. For each even  $i$ with $1\leqslant i \leqslant 2d+1$, there  are unipotent elements  whose corresponding matrix has Jordan canonical form being  the direct sum of $J_i$  and $I_{2d+1-i}$. This gives $d$ distinct classes that are clearly not fused in $\mathrm{P}\Gamma\mathrm{O}(2d+1,q)=\Aut(T)$, as required.

Finally, we have  $T=\POmega^\epsilon(2d,q)$ with $\epsilon\in\{\pm\}$. First suppose that $q$ is odd. For each  $i$ with $1\leqslant i \leqslant d$, there are unipotent elements in $T$ whose corresponding matrix  has Jordan canonical form being  the direct sum of $J_{2i}$ and $I_{2(d-i)}$. This gives $d$ distinct classes that are clearly not fused under inner, diagonal or field automorphisms of $T$.
Moreover, when $2d=8$ and $\epsilon=+$, \cite[Proposition 3.55]{Burness} shows that none of these classes are fused under a triality automorphism. Hence $T$ has at least $d$ $\Aut(T)$-classes of unipotent elements.
Suppose now that  $q$ is even. We again follow  the notation of \cite[Section 8]{aschseitz} for involutions in $T$ and use the geometric description in \cite[Section 3.5.4]{BGbook}.  For each even $i<d$ we get two $T$-classes of involutions with Jordan canonical form consisting of $i$ copies of $J_2$. When $i=d$ and $d$ is even, we get one such $T$-class when $\epsilon=-$ and three when $\epsilon=+$.  Fusing of these classes in $\mathrm{P}\Gamma\mathrm{O}^\pm(2d,q)$  occurs only when $\epsilon=+$ and $i=d$ is even, in which case two of the three classes fuse. When $d$ is odd we have  $\Aut(T)=\mathrm{P}\Gamma\mathrm{O}^\pm(2d,q)$ and so we see that $T$ has at least $d-1$ $\Aut(T)$-classes of involutions as required. Suppose now that $d$ is even. Then the number of  $T$-classes of involutions is   $2(\lfloor (d-1)/2 \rfloor) +3 =  d+1 $ when $\epsilon=+$ and $2(\lfloor (d-1)/2 \rfloor) +1=d$ when $\epsilon=-$. Thus if $\epsilon=-$, or if  $\epsilon=+$ and  $2d\neq 8$, we see  that  $T$ has at least  $d$ $\Aut(T)$-classes of involutions. Finally, if $2d=8$ and $\epsilon=+$, then two  of the classes with $i=4$ (denoted $a_4$ and $a_4'$ in \cite{BGbook}) are fused with one of the classes with $i=2$ (denoted $c_2$ in \cite{BGbook}) under triality \cite[Proposition 3.55]{Burness}. Hence in this final case we also get at least $d-1$  $\Aut(T)$-classes of involutions, as required.
%
\end{proof}

\begin{prop}\label{prop:class}
There is an increasing integer function 
$g$ such that, if $T$ is a finite simple classical group  as in one of the lines of Table~\ref{tab: classical gps}, and $T$ is pp-bounded by $n$, then $|T|<g(n)$.
\end{prop}

\begin{proof} The group $T$ has a natural module $V$ as in Table~\ref{tab: classical gps}, and it is convenient to work with the preimage $G\leqslant \GL(V)$ acting linearly on $V$. We set $q=p^a$ with $p$ prime. 
Let $\varphi$ denote the natural map $\varphi:G \rightarrow T$, and for a subgroup $H$ of $T$ let $\wh{H}$ denote the full preimage of $H$
under $\varphi$. We first observe that, by Lemma~\ref{lem:unipbound}, the number of $\Aut(T)$-classes of unipotent elements of $T$ is at least $d$, and hence    $d\leqslant c_1(n)$ with $c_1(n)=n$. 

There is a prime $m$ satisfying $d/2<m\leqslant d$, and we choose the smallest possible value for $m$, except that we choose $m=3$ if $d=3$. Then one of the following holds:
\begin{enumerate}
    \item[(i)] $m<d$; 
    \item[(ii)] $d=m=2$ with $T=\PSL(2,q)$ or $\PSp(4,q)$;   or 
    \item[(iii)] $d=m=3$ with $T=\PSL(3,q)$,  $\PSU(3,q)$, $\PSp(6,q)$, or $\POmega^\circ(7,q)$.   
\end{enumerate}

\medskip\noindent\emph{Choice of decomposition:} \quad If $m<d$ then we choose a decomposition $V=U\oplus W$, with $\dim(U)=m$ in the linear and unitary cases, and $\dim(U)=2m$ in the symplectic and orthogonal cases  (and possibly $\dim(W)=0$).  Additionally, if $G$ preserves a form, we choose $U$ to be nondegenerate and $W=U^\perp$. Finally, in the orthogonal case, we choose $U$ of minus type. 

\medskip\noindent\emph{Choice of cyclic subgroup:} \quad 
In our arguments we will work with a cyclic subgroup of $G$ that stabilises the decomposition $V=U\oplus W$ and acts trivially on $W$   (see \cite[Section 3]{NieP} for a description of the linear action of these tori). More specifically: 
\begin{enumerate}
    \item[(L)]  if $T=\PSL(d,q)$ we consider a cyclic subgroup  ${H}$ of order $\frac{q^m-1}{(q-1)\cdot (m,q-1)}$ such that $\wh{H}^U$ is a Singer cycle of $\SL(U)$, and $\wh{H}$ fixes $W$ pointwise;
    \item[(U)] if $T=\PSU(d,q)$ then $m$ is an odd prime and we consider a cyclic subgroup  ${H}$ of order $\frac{q^m+1}{(q+1)\cdot (m,q+1)}$ such that $\wh{H}^U$ is a Singer cycle of $\SU(U)$ and $\wh{H}$ fixes $W$ pointwise;
    \item[(Sp)] if $T=\PSp(2d,q)$ we consider a cyclic subgroup  $\wh{H}$ of order $q^m+1$ such that $\wh{H}^U$ is a maximal torus of $\Sp(2m,q)$ and $\wh{H}$ fixes $W$ pointwise;
    \item[(O)]  if $T=\POmega^\epsilon(2d,q)$ or $\POmega^\circ(2d+1,q)$ with $m<d$, and also in the exceptional case where $T=\POmega^\circ(7,q)$ with $q$ odd and $d=m=3$, the parameter $m$ is an odd prime and  $U$ is of minus type and dimension $2m$. By \cite[Lemma 4.1.1(ii)]{KL}, 
    the stabiliser in $G$ of the decomposition $V=U\oplus W$ contains the subgroup $\Omega(U)\times 1\cong\Omega^-(2m,q)$ which fixes $W$ pointwise.  We consider a cyclic subgroup  $\wh{H}$ of $\Omega(U)\times 1$ of order $\frac{q^m+1}{(q+1)_2}$ such that $\wh{H}^U$ is contained in a Singer cycle of $\Omega(U)$. Note that, $|{\rm O}(U):\Omega(U)|=2\cdot (2,q-1)$ (see \cite[Table 2.1.C]{KL}), and that $\frac{q^m+1}{q+1}$ is odd since $m$ is odd, so the $2$-part $(q^m+1)_2=(q+1)_2$ and $|\wh{H}|=(q^m+1)_{2'}$.

    \medskip Thus the cyclic subgroup $\wh{H}$ has the properties given in the appropriate row of Table~\ref{tab:choices for H in all gennric cases}. 
    First we consider the linear case.

\end{enumerate}
\begin{table}[ht]
    \centering
    \begin{tabular}{c|c |c |c |c |l  } \hline 
        Case & $G$ & $|\wh{H}|$ & $\wh{H}^U$ &  $\wh{H}^W$ &  Notes  \\ \hline 
         (L)&  $ \SL(d,q)$ &    $ (q^m-1)/(q-1)$ &  $\leqslant \GL(1,q^m)$ & 1 &    $m=d$ if $d=2$ or $3$   \\ 
         (U)&  $ \SU(d,q)$ &    $ (q^m+1)/(q+1)$ &  $\leqslant \GU(1,q^m)$ & 1 & $m=d$  if $d=3$   \\ 
         (Sp)&  $ \Sp(2d,q)$ &    $ q^m+1 $ &  $\leqslant \Sp(2,q^m)$ & 1 & $m=d$  if $d=2$ or $3$     \\ 
         (O)&  $ \Omega(2d+1,q)$ &    $ (q^m+1)_{2'}$ & $\leqslant \mathrm{GO}^-(2,q^m)$ & 1 &   $m=d$   if $d=3$\\ 
         (O) &  $ \Omega^\pm(2d,q)$ &    $ (q^m+1)_{2'}$ & $\leqslant \mathrm{GO}^-(2,q^m)$ & 1 & \\ \hline 
    \end{tabular}
    \caption{Choices for a cyclic subgroup $\wh{H}$ of $G$}
    \label{tab:choices for H in all gennric cases}
\end{table}



\medskip
\noindent
\emph{Case: $T=\PSL(d,q)$.}\quad 
Here $|\wh{H}|=(q^m-1)/(q-1)$ and we set $H=\varphi(\wh{H})$. Also $H\cong \wh{H}$ if $m<d$, or if $m=d=3$ with $\gcd(3,q-1)=1$, or if $m=d=2$ with $q$ even, since in these cases $\wh{H}$ contains no nontrivial scalar matrix. We assume first that  $|H|=(q^m-1)/(q-1)$, and comment at the end on how to adjust our argument to deal with the exceptional cases when $m=d=2$ or $3$ and $m$ divides $q-1$. 

The centraliser of $\wh{H}^U$  in $\GL(U)$ is the full Singer cycle $C=\GL(1,q^m)$, and $N_{\GL(U)}(\wh{H})=C.m$ (see \cite[Lemma 2.1]{NeuP} or \cite[Satz II.7.3 on page 187]{Hu}); and setting $q=p^a$ with $p$ prime, we have 
\[
N_{\Gamma\mathrm{L}(d,q)}(\wh{H}) = (\GL(1,q^m)\cdot m\times\GL(d-m,q))\cdot a
\]
(or $\GL(1,q^m)\cdot ma$ if $m=d$).
It follows that $r:=|N_{\Aut(T)}(H)/C_{\Aut(T)}(H)|$ divides $2am$ (noting that
in general an outer automorphism corresponding to the ``inverse transpose map''
will act nontrivially on $H$ if $d>2$).

Let $s\in\calS(m)$ as defined in Section~\ref{s:prod}, and let $\wh{S}$ be the Sylow $s$-subgroup of $\wh{H}$ and $S=\varphi(\wh{S})$. 
Since $|\wh{S}|$ does not divide $q-1$ by the definition of $\calS(m)$, we have $|S|=s^b$ for some $b  \geqslant  1$ and $\phi(|S|)=s^{b-1}(s-1)$. 
Also  
$\wh{S}$ is irreducible on $U$, and $S$ and $H$ have the same centraliser and
normaliser in $\Aut(T)$ (see \cite[Lemma 2.1]{NeuP}).   Hence $ |N_{\Aut(T)}(S)/C_{\Aut(T)}(S)|=|N_{\Aut(T)}(H)/C_{\Aut(T)}(H)|=r$ divides $2am$. Further, by Lemma~\ref{lem: bounding phi s}, the number of $\Aut(T)$-classes of elements of $T$ of order $|S|$ is $\phi(|S|)/r$ and divides $n!$. Thus $\phi(|S|)$ divides $r\cdot n!$, which divides 
 $2am\cdot n!$. Since $m\leqslant d$ and $d \leqslant c_1(n)=n$, it follows that  $2m\leqslant 2n$ and hence $2m$ divides $(2n)!$. Thus 
 \begin{equation}
\label{eq:phi s divides}
\phi(|S|)=s^{b-1}(s-1)	\quad \mathrm{divides} \quad c_2(n)a,\quad \text{where} \ c_2(n) = (2n)!\cdot n!.
\end{equation}
Note that the bound $c_2(n)a$ is independent of the prime $s\in\calS(m)$.
The condition \eqref{eq:phi s divides} implies that $s-1$ is a
divisor of $c_2(n)a$, and since the number of divisors of $c_2(n)a$ is less than
$2\sqrt{c_2(n)a}$ (see \cite[Section 8.3]{nzm}), it follows that 
$|\calS(m)| < 2\sqrt{c_2(n)a}$. 

Now $|H| = |\varphi(\wh{H})| = (q^m-1)/ (q-1)$ and recall that $m$ is prime. If  
$s\in\calS(m)$ and $s$ does not divide $q-1$, then
$(q^m-1)_s = |H|_s$ and in this case  by (\ref{eq:phi s divides}) we have
\[
(q^m-1)_s = |S|= s^b = s\phi(|S|)/(s-1) \leqslant  2c_2(n)a.
\]
On the other hand if  $s\in\calS(m)$ and $s$ divides $q-1$, then $s=m$ by Lemma~\ref{lem:ppd}(iv). In this case,  
\[
\frac{(q^m-1)_m}{(q-1)_m}= |S| = m \phi(|S|)/(m-1) \leqslant 2c_2(n)a. 
\]
 Putting this together we see that
\[
\prodS(m,q) = \prod_{s\in \calS(m)} (q^m-1)_s \leqslant (q-1)_m\cdot (2c_2(n)a)^{2\sqrt{c_2(n)a}}\quad \text{so}\quad    \frac{\prodS(m,q)}{(q-1)_m}\leqslant (2c_2(n)a)^{2\sqrt{c_2(n)a}}.
\]
By Lemma~\ref{lem:ppd}(iv), we obtain a lower bound
\[
\frac{\prodS(m,q)}{(q-1)_m}\geqslant  \frac{1}{(q-1)_m} \cdot \frac{q^m-1}{q-1} \geqslant  \frac{q^m-1}{(q-1)^2}. 
\]
Suppose first that $m\geqslant 3$. Then $(q^m-1)/(q-1)^2 > q=p^a$. If $a=1$, this gives the upper bound $q< (2c_2(n))^{2\sqrt{c_2(n)}}$ and we are done. So suppose that $a>1$. Now
\[
(2c_2(n)a)^{2\sqrt{c_2(n)a}} = \left ( (2c_2(n) )^{2\sqrt{c_2(n)}} \right )^{\sqrt{a}} \left (e^{2\sqrt{c_2(n)}}\right )^{\log(a)\sqrt{a}}   \leqslant  c_3(n)^{\sqrt{a}\log a},
\]
for some function $c_3(n)$. Hence 
$a\log p < \sqrt{a}\log a \log(c_3(n))$ which yields $\frac{\sqrt{a}\log p}{\log(a)} < \log (c_3(n)) $. 
Since $a >  1$, we have $\sqrt{a}/\log(a) \geqslant 1$, so $\log(p) \leqslant \log(c_3(n))$ and hence $p\leqslant c_3(n)$. Also, for all $a\geqslant 1$, Lemma~\ref{lem: sqrtx over log x} gives  $\frac{a^{1/4}\log2}{4} \leqslant \frac{\sqrt{a}\log 2  }{\log(a)} \leqslant  \frac{\sqrt{a}\log p} {\log(a)}$, and hence $ a \leqslant \left (\frac{4}{\log2}\log(c_3(n)) \right)^4$.
Thus, when $m\geqslant 3$, we have shown that all of
$a$, $p$ and $d$ are bounded above by functions of $n$, and hence $|T|$ is
also bounded above by some function of $n$. If $m=d=2$ with $q$ even, then $(q-1)_m=1$ and, by Lemma~\ref{lem:ppd}(ii), $p^a<q+1\leqslant \prodS(m,q)$. Our arguments therefore give $p^a< \prodS(m,q) \leqslant c_3(n)^{\sqrt{a}\log a}$, and the same argument yields the required bound.

Now we treat the two exceptional cases: if $m=d=3$ with $3$ dividing $q-1$, then $\wh{H}$ has order $(q^3-1)/(q-1)=q^2+q+1$ and its image $H=\varphi(\wh{H})$ has order $(q^2+q+1)/3$. Also,  by Lemma~\ref{lem:ppd}(iv), $(q^3-1)/(q-1)\leqslant \prodS(m,q)$. Thus the argument above yields
$p^a=q<(q^3-1)/(q-1)^2\leqslant \prodS(3,q)/(q-1)_3 = 3\cdot|H|\leqslant c_3'(n)^{\sqrt{a}\log a}$, and the required bound.

Finally assume that $m=d=2$ with $q$ odd. Here $T$ has unique conjugacy classes of cyclic subgroups of orders $(q+1)/2$ and $(q-1)/2$, and one of these orders is odd, say $(q+\delta)/2$. We choose $H<T$ with $H$ cyclic of odd order $(q+\delta)/2$, and for each prime $s$ dividing $(q+\delta)/2$, we consider the Sylow $s$-subgroup $S$ of $H$. The index $|N_{\Aut(T)}(S) : C_{\Aut(T)}(S)|=2a$, and hence by Lemma~\ref{lem: bounding phi s}, the number of $\Aut(T)$-classes of elements of $T$ of order $|S|$ divides $2a\cdot n!$. Now a very similar argument to that above shows that the product over the primes $s$ of the orders of these Sylow subgroups of $H$ (which equals $|H|$) is bounded 
above by $c_3(n)^{\sqrt{a}\log a}$ for some function $c_3(n)$, and since $|H|\geqslant  (q-1)/2\geqslant  p^a/3$, we conclude that $p$ and $a$, and hence also $q$ are bounded by some function of $n$. Thus $|T|$ is bounded by a function of $n$, which completes the proof in the linear case. 

\medskip



\medskip
\noindent
\emph{Case: The remaining cases with $m=d$ in (ii) and (iii).}\quad 
Here $T$ and $m$ are as in Table~\ref{tab:remaining cases with m=d}, and in each case $T$ contains a 
cyclic subgroup $H$ with order as in the respective row of Table~\ref{tab:remaining cases with m=d}. For each prime $s$ dividing $|H|$, we consider the Sylow $s$-subgroup $S$ of $H$. Then $H$ and $S$ have the same centralisers and normalisers in $\Aut(T)$, and the index $r=|N_{\Aut(T)}(S) : C_{\Aut(T)}(S)|$ is as in Table~\ref{tab:remaining cases with m=d}.
\begin{table}[ht]
    \centering
    \begin{tabular}{c|c | c |c}  
    \hline
         $T$& $m$ & $|H|$ & $r$ \\ \hline
         $\PSp(4,q)$ & $2$ & $(q^2+1)/(2,q-1)$ & $4a\cdot(2,q)$ \\
        $\PSU(3,q)$ & $3$ &  $(q^2-q+1)/(3,q+1)$  & $6a$ \\
         $\PSp(6,q)$ & $3$ & $(q^3+1)/(2,q+1)$& $6a$ \\
         $\POmega^\circ(7,q)$& $3$ &  $(q^3+1)_{2'}$& $6a$ \\ \hline
    \end{tabular}
    \caption{Orders of cyclic subgroups of $T$ in the remaining cases where $m=d$}
    \label{tab:remaining cases with m=d}
\end{table}

By Lemma~\ref{lem: bounding phi s}, the number of $\Aut(T)$-classes of elements of $T$ of order $|S|$ is equal to $\phi(|S|)/r$ and hence $\phi(|S|)$ divides $r\cdot n!$, which divides $24a\cdot n!$ in each of these four cases, and this bound is independent of $s$. Thus in each case 
\[
|S|=s\cdot\frac{\phi(|S|)}{s-1}< 2\phi(|S|)\leqslant 48 a\cdot n!
\]
The number of divisors of $24a\cdot n!$ is less than $2\sqrt{24a\cdot n!}$ by \cite[Section 8.3]{nzm}, so $|H|$, which is the product of $|S|$ over all $s$, satisfies  
\[
|H|\leqslant \left(48 a\cdot n!\right)^{2\sqrt{24a\cdot n!}}.
\]
 On the other hand, in all cases $|H|> q=p^{a}$. If $a=1$, then this bounds $q$ by a function of $n$ and we are done. So suppose that $a>1$. Then 
\[
|H|\leqslant \left(48 a\cdot n!\right)^{2\sqrt{24a\cdot n!}} = \left ( 48\cdot n!)^{2\sqrt{24\cdot n!}} \right )^{\sqrt{a}} \left (e^{2\sqrt{24\cdot n!}}\right )^{\log(a)\sqrt{a}}  \leqslant c(n)^{\sqrt{a}\log a}
\]
for some function $c(n)$. Thus $p^a = q < |H| \leqslant c(n)^{\sqrt{a}\log a}$, and so taking logs we obtain $(\sqrt{a}\log p)/\log a\leqslant \log(c(n))$. Arguing as in the previous case, this bounds both $a$ and $p$, and hence also $q$ and $|T|$,  by some function of $n$.   

\medskip
\noindent 
\emph{Case: The other classical groups with $m<d$.}\quad Here $m$ is an odd prime, $d\geqslant 4$, and the preimage $G$ of $T$ has a cyclic subgroup $\wh{H}$ as in case (U), (Sp), or (O) of Table~\ref{tab:choices for H in all gennric cases}.
Since $m<d$, $\wh{H}$ has nontrivial fixed point space in $V$ and hence contains no nontrivial scalars, so $H:=\varphi(\wh{H})\cong \wh{H}$. In all cases $|H|$ divides $q^m+1$.

By by Lemma~\ref{lem:ppd}(iii) and (iv), $2\not\in\calS(2m)$, and $\prodS(2m,q)$ divides $q^m+1$ and is divisible by $(q^m+1)/(q+1)$. Moreover,  by Lemma~\ref{lem:ppd}(iv)(1), if $s\in\calS(2m)$ then $s$ divides $q^m+1$, and either $(q^m+1)_s = \prodS(2m,q)_s$ divides $(q^m+1)/(q+1)$, or $s=m$ divides $q+1$.  
In the former case we have $(q^m+1)_s=|H|_s = \prodS(2m,q)_s$; while if  $s=m$ divides $q+1$ then, since $m$ is odd, either we again have $(q^m+1)_s=|H|_s = \prodS(2m,q)_s$, or $T=\PSU(m,q)$ and $|H|_m= \frac{(q^m+1)_m}{(q+1)_m}$. 

Now we argue as in the linear case for the primes in $\calS(2m)$ (rather than $\calS(m)$).   Let $s\in\calS(2m)$ and let $\wh{S}$ be the Sylow $s$-subgroup of $\wh{H}$ and $S=\varphi(\wh{S})$ the Sylow $s$-subgroup of $H$. An identical argument to that in the linear case shows that   the number of $\Aut(T)$-classes of elements of $T$ of order $|S|$ is $\phi(|S|)/r$ and divides $n!$, where $r=|N_{\Aut(T)}(H)/C_{\Aut(T)}(H)|$, and for each of the groups in this case $r$ divides $12ma$. The argument thus shows that $\phi(|S|)$ divides $c_2(n)a$ for some function $c_2(n)$, and that there are at most $2\sqrt{c_2(n)a}$ primes in $\calS(2m)$.
For each prime $s$ where  $(q^m+1)_s=|H|_s = \prodS(2m,q)_s$, we therefore have 
\[
(q^m+1)_s=|S| = s\phi(|S|)/(s-1) <2c_2(n)a,
\]
There is at most one prime in $\calS(2m)$ for which this condition fails, namely the prime $s=m$ if $m\mid q+1$ and $T=\PSU(d,q)$. In this exceptional case we have instead that 
\[
\frac{(q^m+1)_m}{(q+1)_m}=|S| = m\phi(|S|)/(m-1) <2c_2(n)a.
\]
Since $s$ is odd, $s$ divides exactly one of \(q^m+1 \)  and \(q^m-1 \). Since $s\in \mathcal S(2m)$, we must therefore have   that $(q^m+1)_s=(q^{2m}-1)_s$.
Thus in the case where $m$ does not arise as an exceptional prime in $\calS(2m)$, we have 
\[
\prodS(2m,q) =\prod_{s\in\calS(2m)} (q^{2m}-1)_s = \prod_{s\in\calS(2m)} (q^m+1)_s 
\leqslant (2c_2(n)a)^{2\sqrt{c_2(n)a}}\leqslant c_3(n)^{\sqrt{a}\log a}
\]
for some $c_3(n)$, and as $\prodS(2m,q)\geqslant  (q^m+1)/(q+1)>q=p^a$, we obtain the required bound on $q$ and $|T|$ as before. 
On the other hand, if $m$ does arise as an exception, so in particular $T=\PSU(d,q)$, then 
we obtain 
\[
\frac{\prodS(2m,q)}{(q+1)_m} =\frac{1}{(q+1)_m} \prod_{s\in\calS(2m)} (q^m+1)_s
\leqslant (2c_2(n)a)^{2\sqrt{c_2(n)a}}\leqslant c'_3(n)^{\sqrt{a}\log a}
\]
for some $c_3'(n)$. Now $\frac{\prodS(2m,q)}{(q+1)_m}\geqslant  (q^m+1)/(q+1)^2\geqslant  (q^3+1)/(q+1)^2>q-2$, and the usual argument yields the required bound. 
\end{proof}

\begin{remark}
    \label{rem:on the function}
    We briefly comment on the function $f(n)$ which appears in Theorem~\ref{t:main}. This  function is the maximum over the functions for different families of simple groups.   From the linear case when $d>2$ and  $q=p$, we have $c_1(n)=n$ so that  $c_2(n)= (2  n)!\cdot n!$. Then \[c_3(n) = 2 ( (2n)!\cdot n!)^{2 \sqrt{(2n)!\cdot n!}}\]
    and so we obtain $\log(p) \leqslant \log (c_3(n)) = \left ({2 \sqrt{(2n)!\cdot n!}}\right ) \log(2 ( (2n)!\cdot n!))$. Since we have $|T| \leqslant p^{d^2}$, we have shown $\log(|T|) \leqslant d^2 \log(p)$ which gives
    \[
    \log(|T|) \leqslant n^2 \left ({2 \sqrt{(2n)!\cdot n!}}\right ) \log(2 ( (2n)!\cdot n!)).
    \]
    The functions arising from the other cases of classical simple groups are similar.
\end{remark}

\begin{remark}
    \label{rem: bounding mexpt for classical}
In  the course of the proof of Proposition~\ref{prop:class} for simple classical groups $T$, we make use of $s$-elements, for carefully chosen primes $s$. One of these primes is the characteristic $s=p$, and for all the other primes $s$, we bound the number of $\Aut(T)$-classes of elements of maximal order $\exp(T)_s = |T|_s$, as discussed in Remark~\ref{rem:mexp}. Thus if $\mathcal{S}(T)$ is the set of these primes $s$, then in our proof of Proposition~\ref{prop:class} we have bounded $|T|$ by a function of the parameter $m'(T)=\max\{m_p(T), m_{\mathcal{S}(T)\text{-exp}}(T)\}$ defined in Remark~\ref{rem:mexp}. We showed in Remark~\ref{rem:mexp-alt} that a similar bound holds for alternating groups. However, our approach to the exceptional groups of Lie type is different from the classical case (see Section~\ref{sec:excep} below), and it is an open question if such a bound holds for the simple exceptional groups of Lie type.


\end{remark}

\section{The exceptional groups}
\label{sec:excep}

In this section we complete the proof of Theorem~\ref{t:main} for the simple groups of Lie type that are exceptional. Our treatment is divided into two cases, depending on whether or not there is a torus subgroup whose normaliser is maximal. Remarkably, most families have this property.

\begin{table}[h]
    \centering
    \begin{tabular}{c|c|c|c |c | l }
    \hline
        $T$ & $|H|$ &   $N_T(H)$ & $z$ & $|\Out(T)|$ & Notes \\ \hline 
        ${}^2B_2(q)$ & ${q\pm \sqrt{2q}+1}$ &  $H.4$ & $4$ & $a$ &  $q=2^{2t+1}$, \cite[Theorem 9]{Suz62}\\
        ${}^2G_2(q)$ & ${q\pm \sqrt{3q}+1}$ &  $H.6$ & $6$ & $a$ & $q=3^{2t+1}$, \cite[Theorem C]{KleidmanG2}\\ 
        ${}^2F_4(q)$ & ${q^2+q+1\pm \sqrt{2q}(q+1)}$ & $H.12$ & $12$ & $a$ & $q=2^{2t+1}$, \cite[Main Theorem]{Malle91} \\
        ${}^3D_4(q)$ & ${q^4-q^2+1}$  & $ H.4 $ & $4$ & $a$ & \cite[Theorem]{Kleidman3D4} \\
        $ F_4(q)$ & $q^4-q^2+1$  & $H.12 $ & $12$ & $2a$ & $q$ even, $q>2$, \cite[Table 8]{Craven2023} \\
        $E_6(q)$ & $(q^2+q+1)/e$   & $({}^3D_4(q) \times H).3 $ & $3$ & $2ae$ & \cite[Table 9]{Craven2023}\\
        ${}^2 E_6(q)$  & $(q^2-q+1)/e'$ & $({}^3D_4(q) \times H).3 $ & $3$ & $ae'$ & $q>2$, \cite[Table 10]{Craven2023}\\
        $ E_7(q)$ & $(q-1)/d$ & $(E_6(q) \times H).2 $ & $2$ & $ad$ & $q\equiv 2 \pmod{3}$, \cite[Table 5.1]{LiebeckSaxlSeitz} \\
        $ E_7(q)$ & $(q+1)/d$ & $({}^2E_6(q) \times H).2 $ & $2$ & $ad$ & $q \not\equiv 2 \pmod{3}$, \cite[Table 5.1]{LiebeckSaxlSeitz} \\
        $E_8(q)$ & $(q^4-1)(q^4-q^3+q)+1$ & $H. 30$ & $30$ & $a$  & \cite[Table 5.2]{LiebeckSaxlSeitz}\\ \hline
    \end{tabular}
    \caption{Choices for a cyclic subgroup $H$ for  exceptional groups $T$ defined  over a field of size $q=p^a$, $p$ a prime,  with $N_T(H)$ maximal,    $d=(2,q-1)$,    $e=(3,q-1)$, $e'=(3,q+1)$}
    \label{tab:easy choices for H for some exceptionals}
\end{table}

\begin{lemma}
\label{lem: easy for exceptional}
%
%
For each simple group $T$ in Table~\ref{tab:easy choices for H for some exceptionals}, there exists a cyclic subgroup $H$ of  the claimed order. Let  $s$ be  prime dividing $|H|$, let $H_s$ be the Sylow $s$-subgroup of $H$ and let $z$ be the corresponding value from Table~\ref{tab:easy choices for H for some exceptionals}. Then one of the following holds:
\begin{enumerate}[$(1)$]
    \item $H_s$ is characteristic in $N_T(H)$, $N_{\Aut(T)}(H)=N_{\Aut(T)}(H_s)$ and $|N_{\Aut(T)}(H) : C_{\Aut(T)}(H_s)|$ divides $z|\Out(T)|$, or
    \item $T=E_7(q)$, $q$ is odd, $s=2$  and there is a nontrivial characteristic cyclic $2$-subgroup $Y_2$ of $N_T(H)$  of order at least $|H|_2/2$ such that  $N_{\Aut(T)}(H) = N_{\Aut(T)}(Y_2)$ and $|N_{\Aut(T)}(H) : C_{\Aut(T)}(Y_2)|$ divides $z|\Out(T)|$, or
    \item $T=E_7(q)$, $q$ is odd, $s=2$   and $|H|_2 \leqslant 2 $.
    \end{enumerate}
\end{lemma}
\begin{proof}
The existence of a cyclic subgroup $H$ of $T$ of  the prescribed order is given by the reference in the Notes column of  Table~\ref{tab:easy choices for H for some exceptionals}. (For $T=E_7(q)$ see also \cite[Table 4.1]{Craven2022}.) Further, each reference proves that $N_{\Aut(T)}(H)$ is a maximal subgroup of $\Aut(T)$.

Let  $s$ be a prime dividing $|H|$ and let $H_s$ be the Sylow $s$-subgroup of $H$. We claim that $s$ is coprime to $z$, except for the case $T=E_7(q)$, $q$ is odd and $s=2$. This is easy to verify, for example, when $T=E_6(q)$ we have $|H|$ is coprime to $3$ unless $q\equiv 1 \pmod{3}$, in which case $q\equiv 1,4,7\pmod{9}$ and so $q^2+q+1\equiv 3\pmod{9}$, which implies that $|H| = (q^2+q+1)/3$ is coprime to $3$; or, for example, when $T=E_8(q)$ then $|H|=(q^4-1)(q^4-q^3+q)+1$ is coprime to $30$.

Suppose first  that  $T=E_7(q)$, $q$ is odd and $s=2$. We may assume that $|H|_2 > 2$ otherwise statement (3) of the Lemma holds.
 If $O_2(N_T(H)) \leqslant H$, then $O_2(N_T(H))=H_2$ and we set $Y_2 = H_2$.  Otherwise, $O_2(N_T(H)) = (H_2).2$. Write $H_2 = \langle y\rangle$. Then since $\langle y^2  \rangle $ is a normal subgroup of $O_2(N_T(H))=\langle y \rangle .2 $ of  index $4$,  we have that  $ \Phi(O_2(N_T(H))) = \langle y \rangle$ or $\langle y^2 \rangle$ (depending on whether $O_2(N_T(H))/ \langle y^2 \rangle$ is cyclic or elementary  abelian, respectively). Hence in this  case we  set $ Y_2 = \Phi(O_2(N_T(H)))$. In both cases,  $Y_2$ is a characteristic cyclic subgroup of $N_T(H)$ of order at least $|H|_2 / 2 $ (in particular, $Y_2\neq 1$).    It follows   that $N_{\Aut(T)}(H) \leqslant N_{\Aut(T)}(Y_2)$ and the maximality of $N_{\Aut(T)}(H)$ in $\Aut(T)$ gives equality. Finally, $C_T(H)$ contains $(N_T(H))'' \cong E_6(q)$ or ${}^2E_6(q)$, and $H\leqslant C_T(H)$, so we have that $|N_{\Aut(T)}(H) : C_{\Aut(T)}(H)|$ divides $z|\Out(T)|$. This proves that statement (2) of the lemma holds.

 Assume now that $T$ is any of the groups in Table~\ref{tab:easy choices for H for some exceptionals}, but that if $T=E_7(q)$ and $q$ is odd, then $s\neq 2$.
Let $H_s$ be the Sylow $s$-subgroup of $H$. Now $N_T(H) = (D \times H).Z$, where $D$ is either trivial or non-abelian simple and $Z$ is a cyclic group of order $z$, as in Table~\ref{tab:easy choices for H for some exceptionals}. 
Since we have excluded the case $s=2$ when $T=E_7(q)$ and $q$ is odd, we have that  $s$ is coprime to $z$. 
It follows that  $H_s=O_s(N_T(H))$, and so $H_s$ is characteristic in $N_T(H)$. Thus  $N_{\Aut(T)}(H) \leqslant N_{\Aut(T)}(H_s)$ and the maximality of $N_{\Aut(T)}(H)$ gives equality.   Finally, since $C_T(H)$ contains $D\times H$, we conclude that $|N_{\Aut(T)}(H) : C_{\Aut(T)}(H)|$ divides $z|\Out(T)|$. This proves that statement (1) of the lemma holds.
\end{proof}

\begin{table}[h]
    \centering
    \begin{tabular}{c|c|c|c |c |c |c }
    \hline
        $T$ & $|H|$ & $ M$ & $N_{T}(H)$ & $z$  & $|\Out(T)|$ & Notes \\ \hline 
        $G_2(q)$ & ${q^2 + q+1}$ & $\SL(3,q):2$ & $H:6$ & $6$ &  $ \leqslant 2a$  & \cite[Theorem 2.3]{Cooperstein81} \cite[Theorem A]{KleidmanG2}\\
        $F_4(q)$ & $q^4 +1$ & $2.\Omega_9(q)$ & $H.4$ & $4$ & $  a$ & $q$ odd, \cite[Table 5.1]{LiebeckSaxlSeitz} , \cite[Table 7]{Craven2023}\\ \hline
    \end{tabular}
    \caption{Choices for a cyclic subgroup $H$ of exceptional groups $T$ when $N_T(H)$ is not maximal, but $N_T(H) < M$ for a unique class of maximal $M$,   $q=p^a$ with $p$ prime.} 
    \label{tab:hard choices for H for some exceptionals}
\end{table}

\begin{lemma}
\label{lem: hard for exceptional}
For each simple group $T$ in Table~\ref{tab:hard choices for H for some exceptionals}, there exists a cyclic subgroup $H$ of  the claimed order. Let  $s$ be a prime dividing $|H|$, let $H_s$ be the Sylow $s$-subgroup of $H$ and let $z$ be the corresponding value in Table~\ref{tab:hard choices for H for some exceptionals}. Then one of the following holds:
\begin{enumerate}[$(1)$]
    \item $H_s$ is characteristic in $N_T(H)$, $N_{\Aut(T)}(H)=N_{\Aut(T)}(H_s)$ and $|N_{\Aut(T)}(H) : C_{\Aut(T)}(H_s)|$ divides $z|\Out(T)|$, or
    \item $T=G_2(q)$,  $q\equiv 1 \pmod{3}$, $s=3$ and   $|H_3| = 3$, or
    \item $T=F_4(q)$,  $s=2$   and    $|H_2| =2$.
\end{enumerate}
%
%

%
\end{lemma}
\begin{proof}
   The proof is similar to that of Lemma~\ref{lem: easy for exceptional}, with adjustments since   $N_T(H)$ is not maximal in $T$. Let $T=G_2(q)$, or $T=F_4(q)$ with $q$ odd. For each row of Table~\ref{tab:hard choices for H for some exceptionals}, take $H$ to be a cyclic subgroup of the specified order inside the specified maximal subgroup $M$ of $T$  (the maximality of $M$ in $T$ is justified by the references in the `Notes' column of Table~\ref{tab:hard choices for H for some exceptionals}). 

 \medskip\noindent
 \textit{Claim:}\quad $N_T(H) = N_M(H)$, and  the structure of the normaliser is as displayed in the respective column of Table~\ref{tab:hard choices for H for some exceptionals}. 

 \medskip\noindent
 \textit{Proof of Claim:} \quad Suppose first that $T=G_2(q)$, and note that $|H|$ has order divisible by a prime $t$ that is a primitive prime divisor    of $q^3-1$ (and  therefore $t \geqslant 5$). Choose  $q_0=p^b$ with $b$ a minimal divisor of $a$ such that $N_T(H) \leqslant G_2(q_0)=:T_0$.  Now, since $T_0$ contains no normal cyclic subgroups, we have $N_T(H) \leqslant L$, for some maximal subgroup $L$  of $T_0$ (and note that $L$ cannot  be a  subfield subgroup by   the  minimality of $q_0$). From the references in the Notes column, we see that $L$ and $t$ are in Table~\ref{tab: cases for L for g2}.
 \begin{table}[h]
     \centering
     \begin{tabular}{c|c | c |c | c}
     \hline
          $L$ & $q_0$ & $t$ & Notes & $q_0^2+q_0+1 \geqslant $ \\ \hline
          $2^3.\PSL(3,2)$ & $p$ & $7$ & $p$ odd &$13$ \\
           $\PSL(2,13)$ & $4$ & $7$ & & $21$ \\
            $J_2$ & $4$ & $7$  & & $21$ \\
            $J_1$ & $11$ & $7,19$ & & $133$\\
            $\PSL(2,13)$ & $q_0=p$  & $7,13$ &   $p\equiv 1,3,4,9,10,12 \pmod{13}$ & $13$\\
            $\PSL(2,13)$ & $q_0=p^2$  & $7,13$ &  $p\equiv 2,5,6,7,8,11 \pmod{13}$  & $21$\\
            $\PSL(2,8)$ & $q_0=p$  & $7$ &   $p\equiv 1,8 \pmod{9}$ & $307$\\
            $\PSL(2,8)$ & $q_0=p^3$  & $7$ &  $p\equiv 2,4,5,7 \pmod{9}$ & $73$  \\
            $\PSU(3,3):2$ & $q_0=p$  & $7$ &   $p\geqslant 5$ & $31$\\
            $\SL(3,q_0):2$ &    & $t$ &    \\ \hline
     \end{tabular}
     \caption{Possibilities for $L$ with $N_T(H) \leqslant L$ for $T=G_2(q)$}
     \label{tab: cases for L for g2}
 \end{table}
Now $H$ is cyclic of order $q^2+q+1 \geqslant q_0^2+q_0+1$, and on comparing the entry for $L$ in Table~\ref{tab: cases for L for g2} with the bound in the last column of that table, we see that $L$  contains a cyclic subgroup of order  $|H|$ only if either $L=\PSL(2,13)$, $t=13$ and $q=q_0=p=3$, or $L=\SL(3,q_0):2$. In the first case, we have $|H|=13=q^2+q+1$ and  $N_{\PSL(2,13)}(H) = H:6$, as in Table~\ref{tab:hard choices for H for some exceptionals}. In the latter case, we have   $|N_{\SL(3,q_0):2}(H)| = (q_0^2+q_0+1).6 \leqslant  | N_M(H)| $, hence $q=q_0$ and $N_T(H) = N_M(H)=H.6$, again as in Table~\ref{tab:hard choices for H for some exceptionals}.

 Suppose now that $T=F_4(q)$, where $q$ is odd, and note that there is a primitive prime divisor $t$ of $q^8-1$ that divides $|H|$ (from which it follows  that  $t \geqslant 11$). As above, choose $q_0=p^b$ with $b$ a minimal divisor of $a$ such that $N_T(H) \leqslant F_4(q_0)=:T_0$, and let  $L$ be a maximal subgroup of $T_0$ containing $N_T(H)$. Then $t $ divides $|L|$, and so from   \cite[Table 1]{Craven2023} and  \cite[Table 7]{Craven2023}  we see that $L$ and $t$ are as in Table~\ref{tab: cases for L for f4}.
 \begin{table}[ht]
     \centering
     \begin{tabular}{c|c }
     \hline
         $L$ & $t$   \\ \hline 
         $\PGL(2,13)$ & $13$  \\
         $\PSL(2,17)$ & $17$  \\
         $\PSL(2,25).2$ & $13$  \\
         $\PSL(2,27).(3)$ & $13$  \\
         $3^3 \rtimes \SL(3,3)$ & $13$ \\
         $2.\Omega_9(q_0)$ & \\ \hline
     \end{tabular}
     \caption{Possibilities for $L$ with $N_T(H) \leqslant L$ for $T=F_4(q)$}
     \label{tab: cases for L for f4}
 \end{table}
 
For the groups in the first $5$ rows of Table~\ref{tab: cases for L for f4}, we have that a subgroup of $L$ of  order $t$ is self-centralising, and thus $|H|=t=q^4+1$. Now, since $q$ is odd, we have $|H|=t=q^4+1 \geqslant q_0^4+1 \geqslant 82$, a contradiction. Hence the  only possibility is that 
$L \cong 2.\Omega_9(q_0)$,  and comparing the orders of $N_{2.\Omega_9(q_0)}(H)$ and $N_M(H)$, we have $q=q_0$  and $N_T(H) = N_M(H)=H.4$ as in Table~\ref{tab:hard choices for H for some exceptionals}. This proves the claim. \hfill $\blacksquare$
    
    


We now turn to the statements of the lemma.    First consider the case $T=G_2(q)$. Then $s$ is coprime to $z$, unless $s=3$ and $q\equiv 1 \pmod {3}$. If $s=3$ and $q\equiv 1 \pmod{3}$, then $q^2+q+1\equiv 3 \pmod{9}$, and so  $|H|_3= (q^2+q+1)_3 = 3$, and statement (2) of the lemma holds. Suppose now that if $s=3$ holds, then $q\not\equiv 1\pmod{3}$. Then $|H|=q^2+q+1$ is coprime to $3$ and odd, so $O_s(N_T(H) ) = H_s$ is a characteristic subgroup of $N_T(H)$, and thus $N_{\Aut(T)}(H)=N_{\Aut(T)}(H_s)$. Arguing similarly to  Lemma~\ref{lem: easy for exceptional}, we see that statement (1) of the lemma holds.

    Consider now the case $T=F_4(q)$. Since $q$ is odd, we have that $q\equiv \pm 1\pmod{4}$, and hence $q^4+1 \equiv 2 \pmod{4}$. Thus $|H|_2=2$ and part (3) of this lemma holds if $s=2$. If $s$ is odd, then $s$ is coprime to $z$ and hence $O_s(N_T(H))=H_s$ satisfies part (1) of the lemma, similarly to the previous cases. 
\end{proof}

\begin{prop}\label{prop:excep}
There is an increasing integer function 
$h$ such that, if $T$ is a simple exceptional  group  of Lie Type, and $T$ is pp-bounded by $n$, then $|T|<h(n)$.
\end{prop}
\begin{proof}
Let $T$ be a simple exceptional group of Lie type defined over the field with $q$ elements. Note that since the rank of $T$ is bounded by a constant, we simply need to prove that $q$ is bounded by a function of $n$. In particular, we may assume that $q>2$ so that we  may  use all rows of Table~\ref{tab:easy choices for H for some exceptionals}.

\medskip\noindent
\emph{Case a): $T$ is one of the groups appearing in Table~\ref{tab:easy choices for H for some exceptionals}.} We apply Lemma~\ref{lem: easy for exceptional} and consider the outcomes in turn.  Let $H$ be the cyclic subgroup of $T$ defined in Lemma~\ref{lem: easy for exceptional} and let $\pi(H)$   be the set of prime divisors of $|H|$. For $s\in \pi(H)$ let $H_s$ be the Sylow $s$-subgroup of $H$. Then, in all cases, we have  
\begin{equation}
\label{eqn: q bounded and |H|}
    \frac{q-1}{2}  \leqslant |H| \quad \text{ and } \quad |H| = \prod_{s \in \pi(H)}|H_s|.
\end{equation}

\medskip\noindent
\emph{Case: a)i): if $T=E_7(q)$, we assume that $q$ is even.}
Now $H_s$ is cyclic and $|H_s|=s^b$ for some $b\geqslant   1$. Let $\mathcal C(s)$ be the set of $\Aut(T)$-classes of elements $g\in T$ such that $g^{\Aut(T)} \cap H \neq \emptyset$ and $|g|=s^b$. Since $T$ is pp-bounded by $n$, it follows from Lemma~\ref{lem: bounding phi s} that $|\mathcal C(s)| $ divides $z|\Out(T)|n!$, which implies that $(s-1)$ divides $360a(n!)$ (using the  lowest common multiple of the entries in the $z$ column of Table~\ref{tab:easy choices for H for some exceptionals}). Since the number of divisors of $360a(n!)$ is at most $2\sqrt{360an!}$, and different values of $s$ give distinct divisors, we have that 
\begin{equation}
\label{eqn: |Pi(H)}
|\pi(H)| \leqslant 2\sqrt{360a(n!)}.
\end{equation} Furthermore, for each $s\in \pi(H)$,  Lemma~\ref{lem: bounding phi s} combined with the value of $|N_{\Aut(T)}(H):C_{\Aut(T)}(H)|$ from Lemma~\ref{lem: easy for exceptional} yields
\begin{equation}
\label{eqn: |H_s|}
    |H_s| = s^b = s^{b-1}(s-1) \left (\frac{s}{s-1}\right ) = \phi(|H_s|) \left (\frac{s}{s-1}\right ) \leqslant 2 z|\Out(T)|n \leqslant 60 a  n.
\end{equation}
Thus putting \eqref{eqn: q bounded and |H|}, \eqref{eqn: |Pi(H)} and \eqref{eqn: |H_s|} together, we have
$$\frac{q}{4} \leqslant \frac{q-1}{2} \leqslant |H| \leqslant |\pi(H)| \max_{s\in \pi(H)}(|H_s|) \leqslant 120 a   n \sqrt{360a(n!)}.$$
This gives
$$\frac{q}{a^{3/2}} \leqslant f(n)$$
with $f(n) = 480n\sqrt{360(n!)}$. By Lemma~\ref{lem: log bound},  
$$\frac{q}{a^{3/2}} = \left (\frac{q}{a} \right )^{3/2}q^{-1/2} \geqslant \left (\frac{\log(2)}{2}q^{1/2}\right )^{3/2}q^{-1/2} = C q^{1/4},$$
for a constant $C$. Thus we obtain $ q^{1/4}  \leqslant  f(n)/C$, which shows that $q$ is bounded by a function of $n$, and hence $|T| \leqslant h_1(n)$, for a computable function $h_1$.

\medskip\noindent
\emph{Case a)ii):  $T=E_7(q)$ and $q$ is odd.} From parts (2.)~and (3.)~of Lemma~\ref{lem: easy for exceptional} we have 
$$
|H| = |H_2|\left (\prod_{s\in \pi(H), \ s \text{ odd}}|H_s|\right) \leqslant \begin{cases}
2\left ( \prod_{s\in \pi(H), \ s \text{ odd}}|H_s|\right) & \text{ if } |H_2|=2,\\
2|Y_2|\left ( \prod_{s\in \pi(H), \ s \text{ odd}}|H_s|\right) & \text{ if } |H_2|>2.
\end{cases}
$$
Since $T$ is pp-bounded by $n$, an identical argument to that given above shows that $|\pi(H)| \leqslant  2\sqrt{z|\Out(T)|(n!)} \leqslant 2\sqrt{4a(n!)}$. Further, for  $s \in \pi(H)$ with $s$ odd, since $T$ is pp-bounded by $n$, Lemma~\ref{lem: bounding phi s} and Lemma~\ref{lem: easy for exceptional} give -- exactly as above -- $|H_s| \leqslant z|\Out(T)|n \leqslant 4an$. For $s=2$, suppose that $|H_2| > 2$, so that Lemma~\ref{lem: easy for exceptional}(2.) holds. Then  Lemma~\ref{lem: bounding phi s} gives $|Y_2|/2 = \phi(|Y_2|) \leqslant  z|\Out(T)|n = 4an$, so that $|Y_2| \leqslant 8an$. Thus 
$$|H| \leqslant 2(8an)\left (2\sqrt{4a(n!)}\right )$$
and since $|H| \geqslant   \frac{q-1}{2}$, we obtain 
$$\frac{q-1}{a^{3/2}} \leqslant 64  n\sqrt{4(n!)}.$$
Similarly to as above, we have $q$ is bounded by a function of $n$, and hence $|T| \leqslant h_2(n)$, for a computable function $h_2$.

\medskip\noindent
\emph{Case b): $T$ appears in Table~\ref{tab:hard choices for H for some exceptionals}.} Here $T=G_2(q)$ or $T=F_4(q)$ with $q$ odd. We apply Lemma~\ref{lem: hard for exceptional} and consider the outcomes in turn.  Let $H$ be the cyclic subgroup of $T$ defined in Lemma~\ref{lem: easy for exceptional} and let $H_s$ be the Sylow $s$-subgroup of $H$ and let $\pi(H)$   be the set of prime divisors of $|H|$.

From parts (2.)~and (3.)~of Lemma~\ref{lem: hard for exceptional}, we may write:
$$|H|=\begin{cases}(3,q-1) \left ( \prod_{s\in\pi(H),\ s\neq 3}|H_s|\right) & \text{if } T=G_2(q)  \\
2 \left ( \prod_{s\in\pi(H),\ s\neq 2}|H_s|\right) & \text{if } T=F_4(q) \end{cases}$$
Now for any $s\in \pi(H)$ such that $(T,s)\ne (G_2(q),3)$ or $(F_4(q),2)$,  
Lemma~\ref{lem: hard for exceptional}(1.) applies, and we argue similarly to the previous cases to obtain $|\pi(H)| \leqslant 2\sqrt{12an!}$. Further, we have $|H_s| \leqslant 2z|\Out(T)|n \leqslant 24an$. This gives
$$q \leqslant |H|/3 \leqslant 2\sqrt{12an!}(24an)$$
and therefore $|T|$ is bounded by a computable  function $h_3(n)$.

\medskip
Finally, we define $h(n)$ to be the maximum over the functions $h_1$, $h_2$ and $h_3$, and thus in any of the cases above, we have $|T| \leqslant h(n)$, as required.
\end{proof}

\section{Proof of Theorem~\ref{t:main}}
 \label{sec:proof}

We set 
\[f(n) = \max \left \{  \frac{1}{2}(3n+2)!,\ g(n),\ h(n) \right \} + |M| \]
where $g(n)$ is the function in Proposition~\ref{prop:class}, $h(n)$ is the function in Proposition~\ref{prop:excep} and $M$ is the Monster sporadic simple group. It is clear from the functions $g$ and $h$ that $f$ is an increasing function. Let  $T$ be a nonabelian finite simple group. If $T$ is sporadic, then $|T| \leqslant |M| \leqslant f(n)$. If $T$ is alternating, classical or an exceptional group of Lie type, then Lemma~\ref{lem: alt}, Proposition~\ref{prop:class} and Proposition~\ref{prop:excep}, respectively, shows that $|T| \leqslant f(n)$. This completes the proof.


\end{document}